\documentclass[11pt,a4paper,twoside]{article}
\usepackage[T1]{fontenc}
\usepackage[utf8]{inputenc}
\usepackage{amsmath,amssymb,amsthm,mathtools}
\usepackage{microtype}
\usepackage[a4paper,textwidth=150mm,textheight=230mm,centering,
            headheight=14pt,headsep=7mm,footskip=10mm]{geometry}
\usepackage{fancyhdr,titlesec,enumitem}
\usepackage[hidelinks,unicode,pdfencoding=auto]{hyperref}
\DeclareUnicodeCharacter{2011}{\nobreakdash-}
\allowdisplaybreaks[2]
\newtheoremstyle{actaplain}{3pt}{3pt}{\itshape}{}{
\bfseries}{}{1em}{}
\newtheoremstyle{actadefinition}{3pt}{3pt}{}{}{\bfseries}{}{1em}{}
\theoremstyle{actaplain}
\newtheorem{theorem}{Theorem}[section]
\newtheorem{lemma}[theorem]{Lemma}
\newtheorem{proposition}[theorem]{Proposition}
\newtheorem{corollary}[theorem]{Corollary}
\theoremstyle{actadefinition}

\newtheorem{example}[theorem]{Example}
\newtheorem{remark}[theorem]{Remark}
\newtheorem{question}[theorem]{Question}
\numberwithin{equation}{section}
\titleformat{\section}{\normalsize\bfseries}{\thesection}{1em}{}
\titleformat{\subsection}{\normalsize\bfseries}{\thesubsection}{1em}{}
\titlespacing*{\section}{0pt}{2.0ex plus .5ex minus .2ex}{1.0ex}
\titlespacing*{\subsection}{0pt}{1.8ex plus .5ex minus .2ex}{.8ex}
\def\ManuscriptAuthor{Puyu  Cui$^{1}$, Ran Li$^{1,\ast}$ and YuFeng Lu$^{2}$}
\def\ManuscriptAffiliation{$^{1}$School of Mathematics, Liaoning Normal University, Dalian 116029,   China\\ $^{2}$School of Mathematical Sciences, Dalian University of Technology, Dalian 116024, China}
\def\ManuscriptEmail{cuipuyu1234@163.com\\liranmika@163.com\\lyfdlut@dlut.edu.cn}
\newcommand{\ShortTitle}{Commuting slant Toeplitz operators}
\newcommand{\ShortAuthors}{P. Cui, R. Li and Y. Lu}
\fancypagestyle{firstpage}{\fancyhf{}\fancyfoot[C]{\small\thepage}
  }
\newcommand{\D}{\mathbb D}
\newcommand{\T}{\mathbb T}
\newcommand{\C}{\mathbb C}
\newcommand{\Z}{\mathbb Z}
\newcommand{\N}{\mathbb N}

\newcommand{\Berg}{A^2(\D)}
\newcommand{\Poly}{\mathcal P_{\infty}}
\DeclareMathOperator{\Hol}{Hol}

\hypersetup{pdftitle={Commuting slant Toeplitz operators with polyanalytic and harmonic symbols},pdfsubject={Acta Mathematica Sinica, English Series: submission-format manuscript},pdfauthor={}}
\begin{document}
\thispagestyle{firstpage}
\begin{center}
{\large\bfseries Commuting slant Toeplitz operators for polyanalytic and harmonic symbols\par}
\ifx\ManuscriptAuthor\empty\else
\vspace{1.6ex}
{\normalsize\ManuscriptAuthor\par}
\fi
\ifx\ManuscriptAffiliation\empty\else
\vspace{.6ex}
{\small\itshape\ManuscriptAffiliation\par}
\fi
\end{center}
\begin{NoHyper}
\begingroup
\renewcommand{\thefootnote}{}
\footnotetext{The third author is supported by the National Natural Science Foundation of China (Grant No.12671149) and the second author is supported by the Educational Foundation of Liaoning Province (Grant No. LJ212610165023).}
\renewcommand{\thefootnote}{\ensuremath{\ast}}
\footnotetext{Corresponding author.}
\endgroup
\end{NoHyper}
\vspace{1ex}
\begingroup
\small
\noindent\textbf{Abstract}\quad
We study commutativity and zero products of slant Toeplitz
operators on the Bergman space of the unit disk.
For bounded polyanalytic symbols and for bounded harmonic
symbols with boundary values in $W^{1,2}(\mathbb T)$,
we prove that two such operators commute if and only if
their symbols are linearly dependent over $\mathbb C$,
with both symbols belonging to the same class.
In the polyanalytic case, no boundedness assumption is
imposed on the individual analytic components.
We also show that every finite-rank commutator in this
class is zero and that every finite product associated
with nonzero polyanalytic symbols has infinite rank.
Finally, for $s\ge2$, we prove that a product of $s$ slant
Toeplitz operators with bounded harmonic symbols whose
boundary values belong to $W^{s-1,2}(\mathbb T)$ is zero
if and only if at least one symbol vanishes identically.
\par
\endgroup
\vspace{1ex}
\section{Introduction and Preliminaries}

Let $\D=\{z\in\C:|z|<1\}$. $\Hol(\D)$ denotes the space of all holomorphic functions on  $\D$.
The Bergman space $\Berg = \Hol(\D) \cap L^2(\D,dA)$ is the closed subspace of analytic functions,  and $L^2(\D,dA)$ is the space consisting of square-integrable functions with normalized area measure $dA(z)=\frac{r\,dr\,d\theta}{\pi}$.

For $f\in L^\infty(\D, dA)$, the Toeplitz operator $T_f$ and the slant Toeplitz operator $B_f$ are
\[
 T_fh=P(fh),\qquad B_fh=WT_fh,
\]
where \[ Wz^{2n}=z^n,\qquad Wz^{2n+1}=0\qquad(n \ge 0)\] and $P$ is the Bergman projection.
Simple computation gives
 $\|B_f\|\le\sqrt2\|f\|_\infty$ and $B_f$ is a bounded operator for $f \in L^\infty (\D, dA)$.
For further details concerning the Bergman space and  Toeplitz operators, see~\cite{Zhu}. A general account of slant Toeplitz operators can be found in \cite{MCHo-1,MCHo-2}.

For $N\in\mathbb{N}$, a function is called  polyanalytic of order at most $N+1$ if  $\overline{\partial}^{N+1}f=0$. Set
\begin{equation}\label{eq:poly-class}
 \Poly=
 \bigcup_{N=0}^{\infty}
 \left\{f\in L^\infty(\D, dA):
 f=\sum_{j=0}^N\overline z^jh_j,
 \ h_j\in\Hol(\D)\right\}.
\end{equation}
That $\Poly$ is a complex vector space follows from the possibility of extending finite representations with zero components to a uniform order.
It should be noted that the individual components \(h_j\) are not required to be bounded.

The core problem for Toeplitz operators lies in the mutual correspondence between operator‑theoretic properties and the properties of symbol functions. Commutativity is a fundamental algebraic property of operators.
 On the Bergman space, commutativity for harmonic symbols was studied by Axler and \v{C}u\v{c}kovi\'c~\cite{AxlerCuckovic}. Brown--Halmos type product results and the range of the Berezin transform were subsequently investigated by Ahern and \v{C}u\v{c}kovi\'c~\cite{AhernCuckovic} and Ahern~\cite{Ahern}. These results illustrate the usefulness of extracting analytic and differential information from operator identities.

Slant Toeplitz operators involve an additional operation $W$. This operation changes the commutativity problem. In particular, $B_1=WT_1=W$, rather than the identity operator $T_1=I$. Lu, Liu and Yang~\cite{LuLiuYang} obtained a linear-dependence characterization for commuting slant Toeplitz operators with bounded analytic symbols. Liu and Lu~\cite{LiuLu} studied products of Slant Toeplitz operators with coanalytic symbols and harmonic polynomial symbols. Wang and Zhao \cite{WangZhao} considered a symbol consisting of a coanalytic polynomial and a bounded analytic function, paired with a bounded analytic symbol. A related common-polynomial-part problem was considered by Zhang \cite{Zhang}.

In section 2, we first consider commuting slant Toeplitz operators with polyanalytic symbols. Theorem~\ref{thm:poly} shows that for any two symbols in this class,
\begin{equation}\label{eq:intro-criterion}
 B_fB_g = B_gB_f \quad \Longleftrightarrow \quad \alpha f+\beta g=0\text{ for some }(\alpha,\beta)\ne(0,0).
\end{equation}
No boundary smoothness is needed for this result.
Next, we characterize the equivalence between the commuting problem for slant Toeplitz operators with these symbols  and the associated finite-rank commuting problem, the infinite-rank property of non-zero products of such operators, and the validity of the related zero-product problem.
In Section 3, we  consider commuting slant Toeplitz operators with  bounded harmonic functions with  radial boundary functions in the Sobolev space $W^{1,2}(\T)$. This assumption includes $C^1(\T)$ and Lipschitz boundary values and is weaker than twice continuous differentiability. Theorem~\ref{thm:har} establishes~\eqref{eq:intro-criterion} for this class as well. A harmonic symbol may have infinitely many $\overline{z}^k~(k \in \N )$, therefore the  radial boundary  functions condition as opposed to the preceding case is needed. The zero product problem  of harmonic symbol  slant Toeplitz operators also holds.

\section{Commuting slant Toeplitz operators with polyanalytic symbols}

In this section, we start with some notation.
Set $ f(z)=\sum_{j=0}^N\overline z^jh_j(z) \in\Poly$ and  write
$ h_j(z)=\sum_{k=0}^{\infty}a_{j,k}z^k$ throughout this section.
Denote $a_{j,k}=0$ for $k<0$. For each integer $t$ and $x$ satisfying $x+t \ge 0$, define the rational
function
\begin{equation}\label{eq:rational}
 R_{f,t}(x)
 =(x+t+1)\sum_{j=0}^N
       \frac{a_{j,t+j}}{x+t+j+1}.
\end{equation}
With these notations, the following lemma holds.
\begin{lemma}\label{lem:rf}
For every integer $n\geq0$, we have
\begin{equation}\label{eq:monomial}
 B_fz^n
 =\sum_{\substack{t\ge \max\{-N, -n\}\\ n+t\text{ even}}}
       R_{f,t}(n)z^{(n+t)/2}.
\end{equation}
Moreover,
\begin{equation}\label{eq:injective}
R_{f,t}\equiv0 \quad\text{ if and only if}\quad a_{j,t+j}=0\quad \text{ for every}\quad 0\leq j\leq N.
\end{equation}
Consequently, if $f\neq0$,
\begin{equation}\label{eq:lowest-index}
 r(f)=\min\{t\in\mathbb Z:R_{f,t}\not\equiv0\}
\end{equation}
is well defined and satisfies $r(f)\geq-N$.
\end{lemma}

\begin{proof}
Since $f\in L^\infty(\D, dA)$, we have $fz^n\overline{z}^m\in L^\infty(\D, dA)$ for every $n,m \ge0$.
Fix $n,m\ge0$ and $0<r<1$. A direct computation implies
\[
 \int_0^{2\pi}
 f(re^{\mathrm{i}\theta})
 e^{\mathrm{i}(n-m)\theta}\,
 \frac{d\theta}{2\pi}
 =
 \sum_{j=0}^N
 a_{j,m-n+j}r^{m-n+2j},
\]
where $a_{j,k}=0$ for $k<0$.
Noting  $fz^n\overline z^{\,m}\in L^1(\D, dA)$, Fubini's theorem gives
\begin{align*}
 \langle fz^n,z^m\rangle
 &=
 2\int_0^1 r^{n+m+1}
 \left(
 \int_0^{2\pi}
 f(re^{\mathrm{i}\theta})
 e^{\mathrm{i}(n-m)\theta}
 \frac{d\theta}{2\pi}
 \right)dr\\
 &=
 2\sum_{j=0}^N
 a_{j,m-n+j}\int_0^1r^{2m+2j+1}\,dr\\
 &=
 \sum_{j=0}^N
 \frac{a_{j,m-n+j}}{m+j+1}.
\end{align*}
Thus the coefficient of $z^m$ in $T_fz^n$ is
\[
 (m+1)\langle fz^n,z^m\rangle
 =
 (m+1)\sum_{j=0}^N
 \frac{a_{j,m-n+j}}{m+j+1}
 =
 R_{f,m-n}(n).
\]
Consequently,
\begin{equation}\label{eq:basis}
 T_fz^n
 =
 \sum_{t\ge\max\{-N,-n\}}
 R_{f,t}(n)z^{n+t},
\end{equation}
with convergence in $A^2(\D)$.
Applying the bounded operator $W$ yields
\eqref{eq:monomial}.

If $R_{f,t}\equiv0$, then $\sum_{j=0}^N\frac{a_{j,t+j}}{n+t+j+1}=0$  for $t\ge \max\{-N, -n\}$ and  $ n+t\text{ even}$.
Fix $t \ge -N$ and set $n+t =2i$, $i \ge 0$. Then $\sum_{j=0}^N\frac{a_{j,t+j}}{2i+j+1}=0$ for $0 \le i, j \le N$. These equations imply
\[
A_t
\begin{pmatrix}
a_{0,\ t} \\
a_{1, \ t+1}\\
\vdots\\
a_{N, t+N}
\end{pmatrix}=0, \qquad A_t =\left(\frac{1}{2i+j+1}\right)_{i, j =0}^N.
\]
Note that the coefficient matrix $A_t$ is a Cauchy matrix, and by  the Cauchy determinant formula,
$$ \mathrm{det} A_t = \frac{\prod_{0\le i<j \le N}(x_j - x_i)(y_j - y_i)}{\prod_{i,j=0}^N(x_i + y_j)}\neq 0,\qquad x_i = 2i,\quad y_j = j + 1.$$
Therefore, $a_{j, \ t+j}=0$ for $j =0, \cdots, N$. The other direction is trivial.

From the above arguments, \eqref{eq:lowest-index} is well defined for nonzero $f$ and $r(f)\geq-N$.
\end{proof}

In the paper,  we write  $(h)_m~(m \ge 0)$ for the $m$-order Taylor coefficient of the analytic function $h$.
\begin{remark}\label{rem:finite-rows}
If $f\neq0$ for all $n>2m-r(f)$, then \eqref{eq:monomial} gives
\[
\left(B_fz^n\right)_m=R_{f,2m-n}(n)=0.
\]
\end{remark}

To prove the main result of this section, we compare coefficients from the power‑series expansions of analytic functions. For this purpose, two lemmas are introduced in advance.

\begin{lemma}\label{lem:ind}
Let $f,g\in\Poly$ be nonzero. If $B_fB_g=B_gB_f$, then
$r(f)=r(g)$.
\end{lemma}

\begin{proof}
Write $r=r(f)$ and $s=r(g)$.  Suppose that $r>s$. Choose sufficiently large $n$  satisfying
\begin{equation}\label{eq:poly-low}
 n+r+2s\equiv0\pmod4.
\end{equation}
Using \eqref{eq:monomial}, the lowest-order nonzero term in the power series expansion of $B_gB_fz^n$ is \[R_{f,r}(n) R_{g,s}\left(\frac{n+r}{2}\right)z^{\frac{n+r+2s}{4}}.\]
Its coefficient is nonzero for sufficiently large $n$, since $ R_{f,r}, R_{g,s} $ are nonzero rational functions.
On the other hand, every nonzero monomial in the  power series expansion of $B_fB_gz^n$ has order at least \[ \frac{n+s+2r}{4} > \frac{n+r+2s}{4}.\]
Thus the coefficient of $z^{\frac{n+r+2s}{4} }$ is zero in $B_fB_gz^n$ but nonzero in $ B_gB_fz^n$.
This contradicts $B_fB_g = B_gB_f$. Hence $r \le s$. Interchanging $f$ and $g$ gives $s \le r$, and therefore $r=s$.
\end{proof}

\begin{lemma}\label{lem:rational-dila}
Let  $r $ be a fixed integer. If a nonzero rational function $R(w)$ for $w\in \mathbb{C}$ satisfies
\begin{equation}\label{eq:rational-inva}
 R\left(\frac{w+r}{2}\right)=R(w),
\end{equation}
then $R$ is constant.
\end{lemma}

\begin{proof}
Notice that $R\left(\frac{(w+r)+r}{2}\right)=R(w+r)=R\left(\frac{w}{2}+r\right)$.
Set $S(w)=R(r+w)$. Then \eqref{eq:rational-inva} becomes
$S(w/2)=S(w)$. Therefore it has a
convergent Laurent expansion in a punctured neighborhood of zero,
\[
 S(w)=\sum_{k=k_0}^{\infty}c_kw^k,
 \qquad k_0\in\mathbb Z,\quad c_{k_0}\neq0.
\]
Comparison of the coefficients yields
$(2^{-k}-1)c_k=0$ for every $k$. It follows that $S$ is a constant function, and hence so is $R$.
\end{proof}

The following theorem is the main conclusion of this section.
\begin{theorem}\label{thm:poly}
Let $f,g\in\Poly$. Then  $B_fB_g=B_gB_f$ if and only if
there exist constants $\alpha$ and $\beta$, not all zero, such that $\alpha f+\beta g =0$.
\end{theorem}

\begin{proof}
It suffices to prove the necessity. Without loss of generality, let $f\neq 0$ and $g \neq 0$.
Lemma~\ref{lem:ind} gives $r(f)=r(g)=r$.
Set nonzero rational functions
\[
 F(x)=R_{f,r}(x),\qquad G(x)=R_{g,r}(x).
\]
 For sufficiently large integers
$n$ satisfying
\begin{equation}\label{eq:poly-equal-congruence}
 n+3r\equiv0\pmod4,
\end{equation}
the lowest possible order in the power expansion of $B_fB_g z^n $ is $(n+3r)/4$.
The same calculation as in the proof of
Lemma~\ref{lem:ind} shows that the corresponding coefficients satisfy
\[
 G(n)F\left(\frac{n+r}{2}\right)
 -F(n)G\left(\frac{n+r}{2}\right)=0
\]
for infinitely many integers $n$. Since a nonzero rational function has only finitely many zeros, we have
\begin{equation}\label{eq:poly-lowest-rational-identity}
 G(x)F\left(\frac{x+r}{2}\right)
 =F(x)G\left(\frac{x+r}{2}\right)
\end{equation}
identically. Rewriting the above identity implies
$F/G$ satisfies \eqref{eq:rational-inva}. By
Lemma~\ref{lem:rational-dila}, there is a nonzero constant
$\lambda$ such that
\begin{equation}\label{eq:poly-lowest-proportionality}
 F=\lambda G.
\end{equation}

Set $h=f-\lambda g$. By adding zero terms to the shorter of the
two representations in $\Poly$, both symbols can be written with the same
 $N$ in \eqref{eq:poly-class}. It follows that
$h\in\Poly$ and
\[
 R_{h,t}=R_{f,t}-\lambda R_{g,t}
 \qquad(t\in\mathbb Z).
\]
By Remark \ref{rem:finite-rows},  both terms on the right of the above equation  vanish for $t<r$. For $t=r$, the difference
vanishes by \eqref{eq:poly-lowest-proportionality}. Statement
\eqref{eq:injective} ensures that the term of index $r$ is cancelled. Consequently, if $h\neq0$, then
\begin{equation}\label{eq:ind2}
 r(h)>r.
\end{equation}
Noting that
\begin{align*}
 B_hB_g-B_gB_h
 &=(B_f-\lambda B_g)B_g-B_g(B_f-\lambda B_g)\\
 &=B_fB_g-B_gB_f=0.
\end{align*}
If $h\neq0$, using Lemma~\ref{lem:ind}, $r(h)=r(g)=r$, contradicting \eqref{eq:ind2}.
Therefore, \eqref{eq:injective} implies $h=0$, so $f=\lambda g$.
\end{proof}

In the following part, denote
\[
 [B_f,B_g]=B_fB_g-B_gB_f.
\]
The same coefficient argument induces the following result.

\begin{corollary}\label{cor:finite-rank}
Let $f,g\in\Poly$. The following conditions are equivalent:
\begin{enumerate}
 \item [(1)] $[B_f,B_g]$ has finite rank;
 \item [(2)]$[B_f,B_g]=0$;
 \item [(3)] $f$ and $g$ are linearly dependent over $\mathbb C$.
\end{enumerate}
\end{corollary}

\begin{proof}
By Theorem~\ref{thm:poly}, it only needs to be shown that  $(1)\Rightarrow (3)$.
We use the obvious fact  that if  $\phi_j (z)~(j \ge 1)$ is analytic and
\[
 \phi_j(z)=c_jz^{m_j}+\sum_{k=m_{j+1}}^\infty c_{j,k}z^k,
 \qquad c_j\neq0,
 \qquad m_1<m_2<\cdots,
\]
then they are linearly independent. 

Suppose that $C =[B_f,B_g]$ has finite rank.  Assume  $f$ and $g$ are nonzero, and write
$r=r(f)$ and $s=r(g)$. We claim that $r=s$. Otherwise, without loss of generality, assume  $r< s$.
For all
sufficiently large $n$ satisfying
$n+2r+s\equiv0\pmod4$, the proof of
Lemma~\ref{lem:ind} shows that $Cz^n$ has lowest
degree term
\[
 R_{g,s}(n)R_{f,r}\left(\frac{n+s}{2}\right) z^{m_n}, \qquad m_n=\frac{n+2r+s}{4}.
\]
As $n$ increases by $4$ in the above equation, $m_n$ increases by $1$. Combining the  above fact gives infinitely many linearly independent vectors in
$\mathrm{Ran}~C$, which is  a contradiction. Hence $r=s$.

Write  $F=R_{f,r}$ and $G=R_{g,r}$, and define the rational
function
\[
 Q(x)=G(x)F\left(\frac{x+r}{2}\right)
       -F(x)G\left(\frac{x+r}{2}\right).
\]
If $Q\not\equiv0$, then $Q(n)\neq0$ for every sufficiently large
integer $n$. For $n+3r\equiv0\pmod4$,  the coefficient of this
degree in $Cz^n$ is exactly $Q(n)$. These degrees again increase by
$1$ as $n$ increases by $4$. This leads to a contradiction with the assumption that $C$ has finite rank. We conclude that $Q\equiv0$.
As in the proof of Theorem~\ref{thm:poly}, we have $F=\lambda G$ and  $\lambda \neq 0$. Let $h=f-\lambda g$. Then
\[
 [B_h,B_g]=[B_f,B_g]=C.
\]
If $h\neq0$,  the lowest-order degree term  gives $r(h)>r(g)$. The preceding proof shows that this is impossible.
Thus $h=0$, and $f=\lambda g$ as desired.
\end{proof}

The polyanalytic theorem contains the nonconstant-polynomial case of Wang and Zhao~\cite[Theorem~3.1]{WangZhao}.

\begin{corollary}\label{cor:wz}
Let $p$ be a nonconstant analytic polynomial and let $\varphi,\psi\in H^\infty(\D)$. Then $[B_{\overline p+\varphi},B_\psi]=0$ if and only if $\psi=0$.
\end{corollary}

The same reasoning gives the common-polynomial-part conclusion discussed in~\cite{Zhang}.
\begin{corollary}\label{cor:common-part}
Let $p$ be a nonconstant analytic polynomial and let $\varphi,\psi\in H^\infty(\D)$. Then $[B_{\overline p+\varphi},B_{\overline p+\psi}]=0$ if and only if $\varphi=\psi$.
\end{corollary}

The lowest-order  term  argument also applies to finite products. We retain the notation $r(f)$ and $R_{f,t}$ from Lemma \ref{lem:rf}.
\begin{proposition}\label{prop:products}
If $f_1,\ldots,f_s\in\Poly\setminus\{0\}$, where $s\ge1$, then
\[
 B_{f_1}\cdots B_{f_s}
\]
has infinite rank. In particular, $f_1,\ldots,f_s\in\Poly$ and \[ B_{f_1}\cdots B_{f_s} =0 \] if and only if there exists $1 \le j \le s$ such that $f_j \equiv 0$.
\end{proposition}
\begin{proof}
Set $r_j=r(f_j)$ and
\[
 \delta=2^{s-1}r_1+2^{s-2}r_2+\cdots+r_s.
\]
For sufficiently large nonnegative integers $n$ with $n+\delta\equiv0\pmod{2^s}$. Beginning with $n_s=n$, define
\[
 n_{j-1}=\frac{n_j+r_j}{2},\qquad j=s,s-1,\ldots,1.
\]
 Iterating the recurrence gives
\[
 n_j=\frac{n+\sum_{k=j+1}^{s}2^{s-k}r_k}{2^{s-j}},
 \qquad 0\le j\le s,
\]
where an empty sum is zero.
It is easy to check $n_j \in \Z$ for $0\le j\le s$.
When  $n$ is sufficiently large, $n_j \ge 0$ and
\[
 n_0=\frac{n+\delta}{2^s}.
\]
For $t_j \ge r_j$, $ B_{f_1}\cdots B_{f_s}z^n$ has related final degree
\[
 \frac{n+2^{s-1}t_1+2^{s-2}t_2+\cdots+t_s}{2^s}.
\]
Its difference from $n_0$ is a sum of nonnegative terms and is zero only when $t_j=r_j$ for every $j$. Consequently, the coefficient of the lowest order possible monomial $z^{n_0}$ is exactly
\[
 \prod_{j=1}^{s}R_{f_j,r_j}(n_j),
\]
and noting that  $n_j$ is a nonconstant affine function of $n$ with positive slope. Since every $R_{f_j,r_j}$ is a nonzero rational function, this product is a nonzero rational function of $n$ and is nonzero for all sufficiently large admissible $n$. Fixed output coefficients involve only finitely many intermediate orders, as  shown in Lemma~\ref{lem:rf}, so the preceding coefficient comparison is valid.
As $n$ increases by $2^s$, $ B_{f_1}\cdots B_{f_s}z^n$ produces a new nonzero lowest-order term. Thus the range is infinite dimensional.
It follows that the zero product problem for slant Toeplitz operators has only the trivial solution.
\end{proof}

\section{Harmonic symbols with Sobolev boundary values}\label{sec:har}

For commuting slant Toeplitz operators with polyanalytic symbols, one may prove the conclusion via the lowest-order term of the analytic function obtained by letting the operators act on $z^n~(n \in \mathbb{N})$. This approach fails for harmonic functions, since their coanalytic components may contain infinitely many monomial terms. Nevertheless, the same conclusion can still be derived by combining the foregoing method with limit-taking arguments.

The following is the main conclusion of this paper.
\begin{theorem}\label{thm:har}
Let $f,g$ be bounded harmonic functions on $\D$, and suppose that their radial boundary  functions $F,G$ belong to $W^{1,2}(\T)$.
Then  $B_fB_g=B_gB_f$ if and only if  $ \alpha f+\beta g=0$ for some $|\alpha|+|\beta| >0$, $ \alpha, \beta \in \C$.
\end{theorem}

We first fix some notation.

For $H\in L^1(\T, d\theta)$, the $j$th Fourier coefficient of $H$ is defined as
\[
 \widehat H(j)=\int_0^{2\pi} H(e^{\mathrm{i}\theta}){e^{-\mathrm{i}j\theta}}\,\frac{d\theta}{2\pi},
 \qquad j\in\mathbb Z.
\]
The Sobolev space $W^{1,2}(\mathbb{T})$ consists of the functions in $L^2(\mathbb{T},d\theta)$ satisfying
\[\sum_{j=-\infty}^{\infty}(1+j^2)|\widehat{H}(j)|^2<\infty.\]
That is, the first weak derivative of $H$ belongs to $L^2(\mathbb{T},d\theta)$.
For $H\in W^{1,2}(\T)$, set
\[
 DH=-\mathrm{i}\frac{dH}{d\theta},\qquad
 D_-H(e^{\mathrm{i}\theta})=\sum_{j=-\infty}^{-1} j\widehat H(j)e^{\mathrm{i}j\theta}.
\]
Furthermore, for an integer $\alpha\ge1$ and
$H\in W^{\alpha,2}(\T)$, define
\[
 D^\alpha H(e^{\mathrm{i}\theta})
 =
 \lim_{i\to\infty}
 \sum_{k=-i}^{i}
 k^\alpha\widehat H(k)e^{\mathrm{i}k\theta}
 \quad\text{in }L^2(\T, d\theta).
\]
Equivalently,
\[
 D^\alpha H(e^{\mathrm{i}\theta})
 =
 (-\mathrm{i})^\alpha
 \frac{d^\alpha}{d\theta^\alpha}H(e^{\mathrm{i}\theta})
\]
in the sense of periodic distributions.
Thus, for $H\in W^{\beta,2}(\T)$,
\[
 D^\alpha H\in W^{\beta-\alpha,2}(\T)
 \qquad(0\le\alpha\le \beta),
\]
and
\[
 D^\alpha D_-H\in W^{\beta-\alpha-1,2}(\T)
 \qquad(0\le\alpha\le \beta-1).
\]
In counting total differential order, $D^\alpha H$
has order $\alpha$, whereas $D^\alpha D_-H$
has order $\alpha+1$.

For  $H \in L^1(\T, d\theta)$, let
\[H^{[2]}(e^{\mathrm{i}\theta})=H(e^{2\mathrm{i}\theta}).\]
The exterior product (wedge product) of $x=\begin{pmatrix}x_1\\x_2\end{pmatrix},y = \begin{pmatrix}y_1\\y_2\end{pmatrix}\in\mathbb C^2$ is denoted by
\[x\wedge y=x_1y_2-x_2y_1.\]

With the assumption that $f,g$ are bounded harmonic functions on $\D$, we write \[F(e^{\mathrm{i}\theta})=\lim\limits_{r \rightarrow 1^{-}} f(re^{\mathrm{i}\theta})=\sum_{j=-\infty}^{\infty}a_je^{\mathrm{i}j\theta},\qquad G(e^{\mathrm{i}\theta})=\lim\limits_{r \rightarrow 1^{-}} g(re^{\mathrm{i}\theta}) =\sum_{j=-\infty}^{\infty}b_je^{\mathrm{i}j\theta}. \]

As with the proof of the main theorem in the preceding section, direct computation yields
\begin{align*}
 \langle fz^n,z^\ell\rangle
 &=\int_0^1 r^{n+\ell+1}dr
       \int_0^{2\pi} f(re^{\mathrm{i}\theta} )e^{\mathrm{i}(n-\ell)\theta}\,\frac{d\theta }{\pi}\\
 &=2a_{\ell-n}\int_0^1r^{n+\ell+|\ell-n|+1}\,dr\\
 &=\frac{a_{\ell-n}}{\max\{n,\ell\}+1},\quad n, \ell\in \N.
\end{align*}
Therefore,
\[
 T_fz^n=\sum_{\ell=0}^\infty
     \min\left\{1,\frac{\ell+1}{n+1}\right\}a_{\ell-n}z^\ell.
\]
Applying $W$ gives
\begin{equation}\label{eq:har-matrix}
 B_fz^n=\sum_{m=0}^\infty w(m,n)a_{2m-n}z^m,
 \qquad
 w(m,n)=\min\left\{1,\frac{2m+1}{n+1}\right\}.
\end{equation}
Further, by $\left([B_f, B_g] z^n\right)_m =0$, we obtain
\begin{equation}\label{eq:har-commutator-matrix}
 \sum_{k=0}^\infty w(m,k)w(k,n)
 \bigl(a_{2m-k}b_{2k-n}-b_{2m-k}a_{2k-n}\bigr)=0
 \qquad(m,n\ge0).
\end{equation}
Fix $d\in\mathbb Z$.  In
\eqref{eq:har-commutator-matrix}, take $N=4m$, $n=N-d$, and set
$k=N/2-j$. Then
\[
 c_{d,j}:=a_jb_{d-2j}-b_ja_{d-2j},\,\text{and}\,\sum_{j=-\infty}^\infty c_{d,j}
\]
is exactly the $d-$ Fourier coefficient of $F(e^{\mathrm{i}2\theta})G(e^{\mathrm{i}\theta})-F(e^{\mathrm{i}\theta})G(e^{\mathrm{i}2\theta})$.
Noting that  $k\ge0$ implies $j\le N/2$.
  Set
\[
 A_j=\min\{j,0\},\qquad C_j=\min\{d-2j,0\},\qquad j, d \in \Z.
\]
From \eqref{eq:har-matrix} and \eqref{eq:har-commutator-matrix},  we conclude that
\begin{equation}\label{eq:har-weighted-identity}
 \sum_{j= -\infty}^\infty\omega_N(j)c_{d,j}=0
\end{equation}
with
\begin{equation}\label{eq:har-omega}
 \omega_N(j)=
 \begin{cases}
 \displaystyle
 \left(1+\frac{2A_j}{N+2-2j}\right)
 \left(1+\frac{C_j}{N+1-d}\right),&j\le N/2,\\[2mm]
 0,&j>N/2.
 \end{cases}
\end{equation}

To obtain the final conclusion, we introduce the following three identities about radial boundary functions $F, G$ from the matrix of the
commutator.  The second-order identity will be needed on the common zero set of $F$ and $G$.
\begin{lemma}\label{lem:har-identities}
Assume the hypotheses of Theorem~\ref{thm:har}, and let
\[
 v=\begin{pmatrix}F\\G\end{pmatrix},\qquad
 u=\begin{pmatrix}D_-F\\D_-G\end{pmatrix}.
\]
If $B_fB_g=B_gB_f$, then the following identities hold almost
everywhere on $\T$:
\begin{align}
 v^{[2]}\wedge v&=0,\label{eq:har-0}\\
 2u^{[2]}\wedge v+v^{[2]}\wedge u&=0,\label{eq:har-1}\\
 u^{[2]}\wedge(Dv+v-u)&=0.\label{eq:har-2}
\end{align}
\end{lemma}

\begin{proof}
From the hypotheses, we obtain
\begin{equation}\label{eq:har-weighted-l2}
 \sum_{j=-\infty}^\infty(1+j^2)|a_j|^2<\infty,
 \qquad
\sum_{j=-\infty}^\infty(1+j^2)|b_j|^2<\infty.
\end{equation}
Using Cauchy--Schwarz  gives
\[
 \sum_{j=-\infty}^\infty|a_j|
 \le
 \left(\sum_{j=-\infty}^\infty(1+j^2)|a_j|^2\right)^{1/2}
 \left(\sum_{j=-\infty}^\infty\frac1{1+j^2}\right)^{1/2}<\infty,
\]
and
\begin{equation}\label{eq:har-weighted-l1}
 \sum_{j=-\infty}^\infty(|a_j|+|b_j|)<\infty.
\end{equation}
Thus,
\begin{equation}\label{eq:har-weak-derivative}
 v\in W^{1,2}(\T;\mathbb{C}^2)\subset C(\T;\mathbb{C}^2),
 \qquad u,Dv\in L^2(\T;\mathbb C^2).
\end{equation}

Define
\begin{align*}
 H_0&=v^{[2]}\wedge v,\\
 H_1&=2u^{[2]}\wedge v+v^{[2]}\wedge u,\\
 J&=u^{[2]}\wedge(Dv+v-u).
\end{align*}
  Direct calculation yields
\begin{equation*}
 H_0=\sum_{d=-\infty}^\infty\left[ \sum_{j=-\infty}^\infty \left(a_jb_{d-2j}-b_ja_{d-2j}\right)\right]e^{\mathrm{i}d\theta}.
\end{equation*}
Similarly, some computation gives for $ d \in \Z$
\begin{align*}
 \widehat H_1(d)&=\sum_{j=-\infty}^\infty(2A_j+C_j)c_{d,j},\\
 \widehat J(d)&=\sum_{j=-\infty}^\infty A_j(d-2j+1-C_j)c_{d,j}.
\end{align*}

Noting that \[
 \sum_{k=0}^\infty|a_{2m-k}b_{2k-n}|
 \le
 \left(\sum_{k=0}^\infty|a_{2m-k}|^2\right)^{1/2}
 \left(\sum_{k=0}^\infty|b_{2k-n}|^2\right)^{1/2}
 \le\|\{a_k\}\|_{\ell^2}\|  \{b_k\}  \|_{\ell^2} < \infty.
\]
Set
$J_d=\max(1,|d|)$.  If $|j|\ge J_d$ and $k=d-2j$, then
\[
 |k|\ge2|j|-|d|\ge|j|.
\]
Combining with   Cauchy--Schwarz yields
\begin{align*}
 \sum_{|j|\ge J_d} j^2|a_jb_{d-2j}|
 &\le\sum_{|j|\ge J_d} |j a_j|\,|(d-2j)b_{d-2j}|\\
 &\le
 \left(\sum_{j=-\infty}^\infty j^2|a_j|^2\right)^{1/2}
 \left(\sum_{j=-\infty}^\infty(d-2j)^2|b_{d-2j}|^2\right)^{1/2}\\
 &\le
 \left(\sum_{j=-\infty}^\infty j^2|a_j|^2\right)^{1/2}
 \left(\sum_{k=-\infty}^\infty k^2|b_k|^2\right)^{1/2}<\infty.
\end{align*}
Consequently,
\begin{equation}\label{eq:har-con-domination}
 \sum_{j=-\infty}^\infty(1+j^2)|c_{d,j}|<\infty.
\end{equation}

Let
\[
 L_d(j)=2A_j+C_j,
 \qquad
 R_N(j)=N^2\left(\omega_N(j)-1-\frac{L_d(j)}N\right).
\]
For $j\le N/2$, \eqref{eq:har-omega} implies
\begin{equation}\label{eq:har-remainder}
R_N(j)=\frac{4A_j(j-1)N}{N+2-2j}
       +\frac{(d-1)C_jN}{N+1-d}
       +\frac{2A_jC_jN^2}{(N+2-2j)(N+1-d)}.
 \end{equation}
Suppose $N\ge2(|d|+1)$.  Then $N+1-d\ge N/2$.  If $A_j\ne0$,
we have $j<0$ and hence $N+2-2j\ge N$; if $A_j=0$, the first and
third terms of \eqref{eq:har-remainder} vanish.  Thus, there is $M_d \ge 0$ such that
\[
 |R_N(j)|
 \le4|A_j||j-1|+2|d-1||C_j|+4|A_jC_j|
 \le M_d(1+j^2)\qquad(j\le N/2),
\]
with $|A_j|\le|j|$ and $|C_j|\le|d|+2|j|$.
For $j>N/2$, we have $A_j=0$ and $C_j=d-2j$, and therefore
\[
 R_N(j)=-N^2-Nd+2Nj.
\]
Since $N<2j$,
\[
 |R_N(j)|\le N^2+N|d|+2Nj
 \le8j^2+2|d|j.
\]
After increasing $M_d$, we conclude that
\begin{equation}\label{eq:har-uniform-remainder}
 |R_N(j)|\le M_d(1+j^2)
 \quad(j\in\mathbb Z,\ N\ge2(|d|+1)).
\end{equation}
For each fixed $j$, formula \eqref{eq:har-remainder} also yields
\begin{equation}\label{eq:har-quadratic-weight}
\lim\limits_{N \to \infty} R_N(j)= Q_d(j), \quad
Q_d(j) :=4A_j(j-1)+(d-1)C_j+2A_jC_j.
\end{equation}
The definition of $R_N(j)$ imples $\omega_N(j)=1+N^{-1}L_d(j)+N^{-2}R_N(j)$ for every $j$, in particular for $j>N/2$.  Thus
\eqref{eq:har-weighted-identity} can be written as
\begin{equation}\label{eq:har-three-orders}
 0=S_0(d)+\frac{S_1(d)}N+\frac{E_N(d)}{N^2},
\end{equation}
where
\[
 S_0(d)=\sum_{j=-\infty}^\infty c_{d,j},\qquad
 S_1(d)=\sum_{j=-\infty}^\infty L_d(j)c_{d,j},\qquad
 E_N(d)=\sum_{j=-\infty}^\infty R_N(j)c_{d,j}.
\]
Noting that  \eqref{eq:har-con-domination} and \eqref{eq:har-uniform-remainder} show that $E_N(d)$ is uniformly bounded.
In the following part, we let   $N$ tend to infinity through multiples of $4$; in particular, $n\ge0$ for all sufficiently large $N$.
Then
\eqref{eq:har-three-orders} gives $S_0(d)=0$, and Identity \eqref{eq:har-0} follows.  Multiplication
by $N$ and letting  $N\to\infty$  in \eqref{eq:har-three-orders}  gives $S_1(d)=0$, and Identity \eqref{eq:har-1} follows.  It follows that $E_N(d)=0$ for all sufficiently large $N$.  Dominated convergence, combining with
\eqref{eq:har-con-domination},
\eqref{eq:har-uniform-remainder}, and
\eqref{eq:har-quadratic-weight}, shows that
\begin{equation}\label{eq:har-coefficient-identities}
\sum_{j=-\infty}^\infty L_d(j)c_{d,j}=0,\qquad
 \sum_{j=-\infty}^\infty Q_d(j)c_{d,j}=0
 \quad(d\in\mathbb Z).
\end{equation}

Write $k=d-2j$,
so that $C_j=\min\{k,0\}$.  A direct calculation gives
\begin{align}
 Q_d(j)-(d-1)L_d(j)
 &=4A_j(j-1)+(d-1)C_j+2A_jC_j\notag\\
 &\hspace{5mm}-(d-1)(2A_j+C_j)\notag\\
 &=2A_j(2j-d-1+C_j)\notag\\
 &=-2A_j(k+1-C_j).
 \label{eq:harmonic-algebraic-cancellation}
\end{align}
The coefficient sums are absolutely convergent by
\eqref{eq:har-con-domination}; thus subtracting
$(d-1)$ times the first equality in
\eqref{eq:har-coefficient-identities} from the second gives
\begin{equation}\label{eq:har-reduced-coeff}
 \sum_{j=-\infty}^\infty A_j\bigl(d-2j+1-C_j\bigr)c_{d,j}=0
 \qquad(d\in\mathbb Z).
\end{equation}
Then Identity \eqref{eq:har-2} holds
\end{proof}

We denote the classical Hardy space by $H^2$.
F. and M. Riesz theorem \cite{Douglas}  indicates that if $h\in H^2$ has radial boundary value zero
on a set of positive measure,  then $h=0$.

\begin{lemma}\label{lem:har-boundary-uniqueness}
Let $f$ be a nonzero bounded harmonic function on $\D$  whose radial boundary function $F \in W^{1,2}(\T)$.  Then the measure of the following set is zero:
\[
\{e^{\mathrm{i}\theta}\in\T: F(e^{\mathrm{i}\theta})=D_-F(e^{\mathrm{i}\theta})=0\}.
\]
\end{lemma}

\begin{proof}
Write
\[
 f=p+\overline q,\qquad
 p(z)=\sum_{n=0}^\infty a_nz^n,\qquad
 q(z)=\sum_{n=1}^\infty \overline{a_{-n}}z^n.
\]
Using \eqref{eq:har-weighted-l1} gives
$p,q\in A(\D)$, where $A(\D)$ denotes the disk algebra.  The
Sobolev condition gives
\[
 \|q'\|_{H^2}^2=\sum_{n=1}^\infty n^2|a_{-n}|^2<\infty.
\]
Writing $q'^*$ for the almost-everywhere radial boundary  function
of $q'$,  it follows that
\begin{equation}\label{eq:harmonic-negative-derivative}
 D_-F(e^{\mathrm{i}\theta})
 =-\sum_{n=1}^\infty n a_{-n}e^{-\mathrm{i}n\theta}
 =-\overline{e^{\mathrm{i}\theta} q'^*(e^{\mathrm{i}\theta})}
 \quad\text{almost everywhere}.
\end{equation}
If $F=D_-F=0$ on a set of positive measure, then $q'=0$.   Since
$q(0)=0$, we have $q=0$.  Then $f=p\in H^2$, whose radial boundary function is $F$.
Applying F. and M. Riesz theorem to $p$ gives $p=0$, contrary
to the assumption $f\ne0$.
\end{proof}

Using the fact that an absolutely continuous function is differentiable almost
everywhere, and this derivative agrees almost everywhere with
its weak derivative, we have the following result.

\begin{lemma}\label{lem:har-ac-zero-set}
Let $v\in W^{1,2}(\T;\mathbb C^2)$ and
$E=\{e^{\mathrm{i}\theta}\in\T: v(e^{\mathrm{i}\theta})=0\}$.  Then \[Dv=0\] almost everywhere on $E$.
\end{lemma}

\begin{proof}
Since the isolated points of $E$ are at most countable, for each isolated point we can choose an arc with rational endpoints that meets $E$ at that point alone.
 Distinct isolated points require distinct such arcs.

Let $e^{\mathrm{i}\theta_0}$ be a nonisolated point of $E$ at which both
components of $v$ are differentiable.  Choose certain distinct $e^{\mathrm{i}\theta_n}\in E$ with
$\theta_n\to\theta_0$.  Since $v(e^{\mathrm{i}\theta_n})=v(e^{\mathrm{i}\theta_0})=0$,
\[
 \frac{d}{d\theta}v(e^{\mathrm{i}\theta})\bigg|_{\theta=\theta_0}
 =\lim_{n\to\infty}
   \frac{v(e^{\mathrm{i}\theta_n})-v(e^{\mathrm{i}\theta_0})}{\theta_n-\theta_0}
 =0.
\]
The excluded nondifferentiability set is null, as is the set of
isolated points.  This finishes the proof.
\end{proof}

\begin{lemma}\label{lem:har-invariant-direction}
Let $V:\T\to\mathbb C^2$ be measurable and nonzero almost everywhere.
If
\[
 V(e^{\mathrm{i}2\theta})\wedge V(e^{\mathrm{i}\theta})=0
 \quad\text{for almost every }e^{\mathrm{i}\theta}\in\T,
\]
then there is a fixed complex line $L\subset\mathbb C^2$ such that
$V(e^{\mathrm{i}\theta})\in L$ almost everywhere.
\end{lemma}

\begin{proof}
For $V(e^{\mathrm{i}\theta})\ne0$, let
\[
 \Pi(e^{\mathrm{i}\theta})=\frac{V(e^{\mathrm{i}\theta})V(e^{\mathrm{i}\theta})^*}{\|V(e^{\mathrm{i}\theta})\|^2},
\]
the rank-one orthogonal projection onto $\mathbb CV(e^{\mathrm{i}\theta})$.
Define it to be an arbitrary fixed rank-one projection on the
null set where $V=0$.  Its matrix entries are bounded measurable
functions.
Since
\[
 \frac1{2\pi}\int_0^{2\pi}\varphi(e^{2\mathrm{i}\theta})\,d\theta
 =\frac1{2\pi}\int_0^{2\pi}\varphi(e^{\mathrm{i}\theta})\,d\theta
 \]
for every integrable $\varphi$,  $V(e^{\mathrm{i}\theta})$ and
$V(e^{2\mathrm{i}\theta})$ are both nonzero almost everywhere.  The determinant
condition says that they span the same line, and therefore
\[
 \Pi(e^{\mathrm{i}2\theta})=\Pi(e^{\mathrm{i}\theta})\quad\text{almost everywhere}.
\]

Define  $T(e^{\mathrm{i}\theta})=e^{\mathrm{i}2\theta}, e^{\mathrm{i}\theta} \in \T$.
Every $h\in L^\infty(\T, d\theta)$ satisfying $h\circ T=h$ is constant almost everywhere.  Indeed,
\begin{align*}
 \widehat{h\circ T}(n)
 &=\frac1{4\pi}\int_0^{4\pi}h(e^{\mathrm{i}t})e^{-\mathrm{i}nt/2}\,dt\\
 &=\frac{1+(-1)^n}{4\pi}
      \int_0^{2\pi}h(e^{\mathrm{i}t})e^{-\mathrm{i}nt/2}\,dt.
\end{align*}
It follows that
\[
 \widehat h(2k+1)=0,\qquad
 \widehat h(2k)=\widehat h(k)\qquad(k\in\mathbb Z).
\]
Since every nonzero integer is $2^ri$ with $i$ odd,  through a repetitive process, $\widehat h(n)=0$ for every $n\ne0$. Thus, $h$ is constant almost everywhere.
Applying this to each matrix entry gives a fixed rank-one
orthogonal projection $\Pi_0$ with $\Pi=\Pi_0$ almost everywhere.
Since $\Pi V=V$, the desired line is $L=\operatorname{Ran}\Pi_0$.
\end{proof}

We now prove the main theorem of this paper.

\begin{proof}[Proof of Theorem \ref{thm:har}]
It suffices to prove necessity. If $f=0$, the conclusion holds, so assume $f\ne0$.  Use $v,u$ from
Lemma~\ref{lem:har-identities} and let
\[
 E=\{e^{\mathrm{i}\theta}\in\T:v(e^{\mathrm{i}\theta})=0\}.
\]
Lemma~\ref{lem:har-ac-zero-set} gives $Dv=0$ almost
everywhere on $E$.  Choose measurable versions of $u$ and $Dv$.
All identities in Lemma~\ref{lem:har-identities} hold outside
a null set; the inverse image of any further null set under the
measure-preserving map $e^{\mathrm{i}\theta} \mapsto  e^{\mathrm{i}2\theta}$ is also null.

Define
\[
 V(e^{\mathrm{i}\theta})=
 \begin{cases}
 v(e^{\mathrm{i}\theta}),&e^{\mathrm{i}\theta}\notin E,\\
 u(e^{\mathrm{i}\theta}),&e^{\mathrm{i}\theta}\in E.
 \end{cases}
\]
We claim that $V(e^{\mathrm{i}2\theta})\wedge V(e^{\mathrm{i}\theta})=0$ almost everywhere.
There are four cases.  If neither $e^{\mathrm{i}\theta}$ nor $e^{2\mathrm{i}\theta}$ lies in
$E$, the claim is \eqref{eq:har-0}.  If $e^{\mathrm{i}\theta}\in E$ but
$e^{\mathrm{i}2\theta}\notin E$, then \eqref{eq:har-1} reduces to
$v(e^{\mathrm{i}2\theta})\wedge u(e^{\mathrm{i}\theta})=0$.  If $e^{\mathrm{i}\theta}\notin E$ but
$e^{\mathrm{i}2\theta}\in E$, the same identity reduces to
$2u(e^{\mathrm{i}2\theta})\wedge v(e^{\mathrm{i}\theta})=0$.  Finally, if both points lie in
$E$, then Lemma \ref{lem:har-ac-zero-set} and
\eqref{eq:har-2} give
\[
 u(e^{\mathrm{i}2\theta})\wedge\bigl(-u(e^{\mathrm{i}\theta})\bigr)=0
\]
outside a null set.  This proves the claim.

By definition $V\ne0$ on $E^c$.  If $e^{\mathrm{i}\theta}\in E$ and
$V(e^{\mathrm{i}\theta})=0$, then $F(e^{\mathrm{i}\theta})=D_-F(e^{\mathrm{i}\theta})=0$.
Lemma~\ref{lem:har-boundary-uniqueness}, applied to the
nonzero function $f$, shows that this exceptional set is null.
Thus $V\ne0$ almost everywhere, and
Lemma~\ref{lem:har-invariant-direction} gives a fixed line
$L$ containing $V(e^{\mathrm{i}\theta})$ almost everywhere.

Choose a nonzero vector $c=(c_1,c_2)^{\mathsf T}$ spanning $L$ and
set $\alpha=-c_2$, $\beta=c_1$.  Then
$(\alpha,\beta)\ne(0,0)$ and
$\alpha V_1+\beta V_2=0$ almost everywhere.  On $E^c$, this is
$\alpha F+\beta G=0$; on $E$, the latter equality holds because
$F=G=0$.  Consequently the boundary function
$\alpha F+\beta G$ vanishes almost everywhere.  Its Poisson
integral is
\[
 \alpha f(z)+\beta g(z)
 =\int_0^{2\pi}\frac{1-|z|^2}{|e^{\mathrm{i}\theta}-z|^2}
       \bigl(\alpha F(e^{\mathrm{i}\theta})+\beta G(e^{\mathrm{i}\theta})\bigr)\,\frac{d\theta}{2\pi}
 =0\qquad(z\in\D),
\]
as desired.
\end{proof}

\begin{remark}\label{rem:harmonic-regularity}
Theorem~\ref{thm:har} includes the $C^2(\T)$ case, and in fact
requires only square-integrable first derivatives.  The
key is the algebraic cancellation
\eqref{eq:harmonic-algebraic-cancellation}: although second-order
matrix asymptotics are used, no second boundary derivatives are
needed.  The theorem allows a common boundary zero set of
positive measure.  On that set the third identity in
Lemma~\ref{lem:har-identities} supplies the missing relation.

The Sobolev condition is strictly weaker than $C^2(\T)$ smoothness;
for example, the periodic function $F(e^{\mathrm{i}\theta})=|\theta|$ for
$-\pi\le\theta\le\pi$ belongs to $W^{1,2}(\T)$ and is not $C^1(\T)$.
Nevertheless, every $W^{1,2}(\T)$ function has a continuous
representative.  Thus the theorem does not remove all boundary
continuity, and makes no assertion for arbitrary bounded harmonic
symbols with only $L^\infty(\T, d\theta)$ boundary values.
\end{remark}

Let $f$ be a bounded harmonic functions on $\D$. For $0<r<1$,
define $f_r(z)=f(rz)$.  It is easy to check that $\|f_r-f\|_{L^2(\D,dA)}\longrightarrow 0\quad(r \to 1^{-})$.
By this approximation, we obtain the following conclusion.
\begin{proposition}\label{prop:har-radial-approximation}
Suppose that $f,g$ are bounded harmonic functions on $\D$. If $\|f_r-f\|_{L^2(\D,dA)}\longrightarrow0,$ then
\[ \left\|(B_{f_r}-B_f)h\right\|\longrightarrow0, \qquad
\left\| [B_{f_r},B_{g_r}] h\right\| \longrightarrow \left\|[B_f,B_g]h\right\|\]
for every $ h \in \Berg$.
\end{proposition}

\begin{proof}
The verification is straightforward, and we omit the details.
\end{proof}

\begin{remark}\label{rem:harmonic-dilation-caveat}
Radial boundary approximation alone does not extend
Theorem~\ref{thm:har} to arbitrary bounded harmonic symbols.
For each $r<1$, the symbols $f_r,g_r$ are harmonic on a disk
strictly larger than $\overline\D$, so their boundary values are
smooth.  However, if $[B_f,B_g]=0$, Proposition~\ref{prop:har-radial-approximation}
only gives $[B_{f_r},B_{g_r}]\to0$ strongly.  Applying
Theorem~\ref{thm:har} to a dilated pair would require
exact commutativity for that pair, which the convergence
statement alone does not establish.

The matrix entries identify the additional assertion that would
have to be proved.  Write $a_j,b_j$ for the boundary Fourier
coefficients of $f,g$; they belong to $\ell^2(\mathbb Z)$ even
without Sobolev regularity.  The dilated coefficients are
$r^{|j|}a_j,r^{|j|}b_j$, and the harmonic matrix calculation gives
\begin{equation}\label{eq:harmonic-dilated-matrix}
 \begin{split}
 \left([B_{f_r},B_{g_r}]z^n\right)_m
 =\sum_{k=0}^\infty&w(m,k)w(k,n)
       r^{|2m-k|+|2k-n|}\\[-1mm]
 &\times\bigl(a_{2m-k}b_{2k-n}-b_{2m-k}a_{2k-n}\bigr).
 \end{split}
\end{equation}
The series is absolutely convergent by Cauchy--Schwarz,
as in the proof of \eqref{eq:har-commutator-matrix}.  In
particular, taking $m=n=0$ yields
\begin{equation}\label{eq:harmonic-dilated-constant-entry}
 \langle[B_{f_r},B_{g_r}]1,1\rangle
 =\sum_{k=1}^\infty\frac{r^{3k}}{k+1}
       \bigl(a_{-k}b_{2k}-b_{-k}a_{2k}\bigr).
\end{equation}
At $r=1$ the corresponding series is zero for a commuting pair.
Convergence of the weighted series to that value does not by
itself establish that it is zero for every $r<1$.  This observation
is a limitation of the approximation argument, not a claimed
counterexample to preservation of commutativity for harmonic
symbols.  The Sobolev extension above instead uses the original,
exact matrix identities and proves the required limits directly.
\end{remark}

The next example explains why the higher-order identities in Section~\ref{sec:har} are needed even for smooth boundary data.
\begin{example}\label{ex:common-zero}
Let
\[
 I=\{e^{\mathrm{i}\theta}:\pi/2<\theta<\pi\},
\]
and choose two nonzero smooth functions $F,G$ with disjoint compact supports contained in $I$. For instance, set
\[
 \rho(t)=
 \begin{cases}
 \exp\!\bigl(-1/(1-t^2)\bigr),& |t|<1,\\
 0,& |t|\ge1,
 \end{cases}
\]
and, on $0\le\theta\le2\pi$, take
\[
 F(e^{\mathrm{i}\theta})=\rho\!\left(\frac{16}{\pi}\left(\theta-\frac{5\pi}{8}\right)\right),
 \qquad
 G(e^{\mathrm{i}\theta})=\rho\!\left(\frac{16}{\pi}\left(\theta-\frac{7\pi}{8}\right)\right).
\]
Their periodic extensions are $C^\infty(\T)$, and their Poisson extensions $f,g$ are bounded harmonic functions. Whenever $e^{\mathrm{i}\theta}\in I$, one has $e^{\mathrm{i}2\theta}\notin I$. Thus
\[
 F(e^{\mathrm{i}2\theta})G(e^{\mathrm{i}\theta})=G(e^{\mathrm{i}2\theta})F(e^{\mathrm{i}\theta})=0
 \qquad(e^{\mathrm{i}\theta}\in\T).
\]
The leading boundary identity is therefore satisfied identically. Nevertheless, $F$ and $G$ are linearly independent because their supports are disjoint and both are nonzero. Their Poisson extensions are consequently linearly independent. Theorem~\ref{thm:har} gives $[B_f,B_g]\ne0$. This example also has a common boundary zero set of positive measure.
\end{example}

\begin{remark}
The two main theorems cover different symbol classes. The polyanalytic theorem requires boundedness of the full symbol, while its analytic components may be unbounded. The harmonic theorem allows $\overline{z}^k~(k\ge 1)$ and common boundary zero sets of positive measure under the stated $W^{1,2}(\T)$ hypothesis, which still provides continuous representatives. Neither theorem settles commutativity for arbitrary bounded harmonic symbols. The additional assertion needed for the radial boundary approximation argument is identified in Remark~\ref{rem:harmonic-dilation-caveat}.
\end{remark}

The case $s=1$ of the following theorem without boundary condition comes from Proposition 2.1 in \cite{WangZhao}.  Now,  we  consider the case $s\ge 2$.
\begin{theorem}\label{thm:har-zero}
Let $s\ge2$, and let $f_1,\ldots,f_s$ be bounded harmonic functions
on $\D$ whose radial boundary functions $F_1,\ldots,F_s$ belong to
$W^{s-1,2}(\T)$. Then $B_{f_1}\cdots B_{f_s}=0$ if and only if
\[ f_j\equiv0\]
for some $1 \le j \le s$.
\end{theorem}
\begin{proof}
It suffices to prove necessity. Suppose that
$A=B_{f_1}\cdots B_{f_s}=0$, and set
\[
 q=s-1,\qquad Q=2^s,\qquad d_j=2^{s-j},\qquad
 H_j(e^{\mathrm{i}\theta})=F_j(e^{\mathrm{i}d_j\theta}).
\]
Write
\[
 F_j(e^{\mathrm{i}\theta})
 =\sum_{k=-\infty}^\infty a_k^{(j)}e^{\mathrm{i}k\theta}.
\]
The Sobolev hypothesis gives
\begin{equation}\label{eq:har-zp-sobolev}
 \sum_{k=-\infty}^\infty (1+|k|^{2q})|a_k^{(j)}|^2<\infty,
 \qquad \{a_k^{(j)}\}\in\ell^1.
\end{equation}

Fix $\ell\in\Z$. For $\mathbf k=(k_1,\ldots,k_s)\in\Z^s$, set
\[
 L_j=\sum_{i=1}^j2^{j-i}k_i,\qquad x_0=N,\qquad
 x_j=2^jN-L_j=2x_{j-1}-k_j.
\]
When $\sum_{j=1}^s d_jk_j=\ell$, we have $x_s=QN-\ell$.
For all sufficiently large  integers $N >0$, iteration of
\eqref{eq:har-matrix} gives
\begin{equation}\label{eq:har-zp-exact}
 0=(Az^{QN-\ell})_N
 =\sum_{\sum_jd_jk_j=\ell}\omega_N(\mathbf k)
                         \prod_{j=1}^s a_{k_j}^{(j)},
\end{equation}
where
\[
 \omega_N(\mathbf k)=
 \begin{cases}
 \displaystyle\prod_{j=1}^s
 \left(1+\frac{\min\{k_j,0\}}{x_j+1}\right),
       &x_1,\ldots,x_{s-1}\ge0,\\[2mm]
 0,&\text{otherwise}.
 \end{cases}
\]
In particular, $0\le\omega_N\le1$. The absolute summability in
\eqref{eq:har-zp-sobolev} justifies \eqref{eq:har-zp-exact} by
uniform approximation of the symbols by harmonic polynomials.

Let $t=(QN)^{-1}$. For each fixed $\mathbf k$, the intermediate
indices are eventually nonnegative, and
\[
 \omega_N(\mathbf k)=
 \prod_{j=1}^s
 \left(1+\frac{d_j\min\{k_j,0\}t}
                  {1-d_j(L_j-1)t}\right).
\]
The coefficients of this power series are $p_0=1$ and
\begin{equation}\label{eq:har-zp-pr}
 p_r(\mathbf k)=
 \sum_{\substack{\nu_1,\ldots,\nu_s\ge0\\
                 \nu_1+\cdots+\nu_s=r}}
 \prod_{\{j:\nu_j>0\}}
 d_j^{\nu_j}\min\{k_j,0\}(L_j-1)^{\nu_j-1}
 \qquad(r\ge1).
\end{equation}
If $M=1+\sum_{j=1}^s |k_j|$, then
\[
 |p_r(\mathbf k)|\le C_{s,q}M^r\qquad(0\le r\le q)
\]
and
\begin{equation}\label{eq:har-zp-remainder}
 t^{-q}\left|\omega_N(\mathbf k)
       -\sum_{r=0}^{q-1}p_r(\mathbf k)t^r\right|
 \le C_{s,q}M^q.
\end{equation}
Indeed, if $N\ge2M$, all intermediate indices are positive and
\[
 \frac{|L_j-1|}{2^jN}\le\frac{M}{2N}\le\frac14.
\]
The estimate follows by expanding the geometric series and bounding
the terms of degree at least $q$. If $N<2M$, it follows from
$0\le\omega_N\le1$, the preceding bounds for $p_r$, and
$t^{-1}=QN<2QM$. For each fixed $\mathbf k$, the expression in
\eqref{eq:har-zp-remainder}, without absolute values, tends to
$p_q(\mathbf k)$.

For every fixed $\ell$, one also has
\begin{equation}\label{eq:har-zp-summability}
 \sum_{\sum_jd_jk_j=\ell}
 M^q\prod_{j=1}^s|a_{k_j}^{(j)}|<\infty.
\end{equation}
To see this, use $M^q\le C_{s,q}(1+\sum_j|k_j|^q)$.
For a term involving $|k_j|^q$, fix all indices except $k_j$
and $k_h$, where $h\ne j$. The remaining indices satisfy
$d_jk_j+d_hk_h=b$. Both coordinate projections of the solution
set are injective, so Cauchy--Schwarz bounds the corresponding
sum by
\[
 \left\|\left\{|k|^q a_k^{(j)}\right\}\right\|_{\ell^2}
 \left\|\left\{a_k^{(h)}\right\}\right\|_{\ell^2}.
\]
Summing over the other indices and using their $\ell^1$ norms
proves \eqref{eq:har-zp-summability}.

Set
\[
 S_r(\ell)=\sum_{\sum_jd_jk_j=\ell}
                p_r(\mathbf k)\prod_{j=1}^s a_{k_j}^{(j)}.
\]
By \eqref{eq:har-zp-remainder}, \eqref{eq:har-zp-summability},
and dominated convergence, \eqref{eq:har-zp-exact} becomes
\[
 0=\sum_{r=0}^q\frac{S_r(\ell)}{(QN)^r}+o(N^{-q}).
\]
Letting $N\to\infty$ and then multiplying successively by powers
of $QN$ gives
\begin{equation}\label{eq:har-zp-vanishing}
 S_r(\ell)=0\qquad(0\le r\le q,\ \ell\in\Z).
\end{equation}

We next identify these coefficient identities on $\T$.
For trigonometric polynomials, expansion of
\eqref{eq:har-zp-pr} shows that $S_r(\ell)$ is the $\ell$-th
Fourier coefficient of a finite sum $U_r$ of products whose
factors have the form $D^\alpha H_j$ or
$D^\alpha D_-H_j$, with total differential order at most
$r\le q$. Every factor of order at most $q-1$ is bounded.
At most one factor has order $q$, and it belongs to $L^2(\T, d\theta )$.
These products therefore depend continuously on
$(F_1,\ldots,F_s)\in(W^{q,2}(\T))^s$ with values in $L^2(\T, d\theta)$.
Approximation by trigonometric polynomials in $W^{q,2}$, together
with \eqref{eq:har-zp-summability}, gives $U_r\in L^2(\T,d\theta)$ with
\[
 \widehat U_r(\ell)=S_r(\ell).
\]
Thus \eqref{eq:har-zp-vanishing} and uniqueness of Fourier
coefficients yield
\begin{equation}\label{eq:har-zp-Ur-zero}
 U_0=\cdots=U_q=0\quad\text{a.e. on }\T.
\end{equation}
In particular,
\[
 U_0(e^{\mathrm{i}\theta})
 =\prod_{j=1}^sH_j(e^{\mathrm{i}\theta})=0
 \qquad(e^{\mathrm{i}\theta}\in\T),
\]
since each $H_j$ is continuous.

Assume that all $f_j$ are nonzero, and let $m$ be the maximum
number of nonzero values among
$H_1(e^{\mathrm{i}\theta}),\ldots,H_s(e^{\mathrm{i}\theta})$.
Then $1\le m\le s-1$. Choose a point attaining this maximum.
The corresponding $m$ functions remain nonzero on a nonempty
open arc $I$ about that point. By maximality, the remaining
$r=s-m$ functions vanish identically on $I$. Let $J$ be their
index set, so that $|J|=r\le q$.

Consider \eqref{eq:har-zp-pr} at this order $r$. If $\nu_j=0$
for some $j\in J$, the corresponding differential expression
contains only an ordinary derivative of $H_j$ in that factor;
it vanishes on $I$. A nonzero contribution on $I$ therefore
requires $\nu_j\ge1$ for every $j\in J$. Since
$\sum_{j=1}^s \nu_j=r=|J|$, the only possible contribution has
$\nu_j=1$ for $j\in J$ and $\nu_j=0$ otherwise. Consequently,
using
\[
 D_-H_j(e^{\mathrm{i}\theta})
 =d_j(D_-F_j)(e^{\mathrm{i}d_j\theta}),
\]
we obtain
\[
 U_r(e^{\mathrm{i}\theta})
 =\left(\prod_{j\in J}D_-H_j(e^{\mathrm{i}\theta})\right)
  \left(\prod_{j\notin J}H_j(e^{\mathrm{i}\theta})\right)
 \quad\text{for a.e. }e^{\mathrm{i}\theta}\in I.
\]
By \eqref{eq:har-zp-Ur-zero} and the nonvanishing of the second
product on $I$,
\begin{equation}\label{eq:har-zp-local-contradiction}
 \prod_{j\in J}D_-H_j(e^{\mathrm{i}\theta})=0
 \quad\text{for a.e. }e^{\mathrm{i}\theta}\in I.
\end{equation}

For each $j\in J$, $H_j$ is the boundary function of the nonzero
bounded harmonic function $z\mapsto f_j(z^{d_j})$, and
$H_j\in W^{q,2}(\T)\subset W^{1,2}(\T)$.
Lemma~\ref{lem:har-boundary-uniqueness} shows that
\[
 \{e^{\mathrm{i}\theta}\in\T:
 H_j(e^{\mathrm{i}\theta})=D_-H_j(e^{\mathrm{i}\theta})=0\}
\]
is a null set. Since $H_j=0$ on $I$, every
$D_-H_j$ with $j\in J$ is nonzero almost everywhere on $I$.
Their finite product is therefore nonzero almost everywhere
there, contradicting \eqref{eq:har-zp-local-contradiction}.
Hence $f_j\equiv0$ for some $j$. The proof is finished.
\end{proof}

\begin{question}
Is it possible to remove the boundary conditions in the theorems  \ref{thm:har} and \ref{thm:har-zero}?
\end{question}
Up to now,  we have not found the counterexamples.


\par\bigskip
\noindent
\begin{minipage}{\linewidth}
\small
\textbf{E-mail addresses:}\\[0.3ex]
\texttt{\ManuscriptEmail}
\end{minipage}
\par

\end{document}